\documentclass[sn-mathphys,Numbered]{sn-jnl}

\usepackage{graphicx}%
\usepackage{multirow}%
\usepackage{amsthm}%
\usepackage{mathrsfs}%
\usepackage[title]{appendix}%
\usepackage{xcolor}%
\usepackage{textcomp}%
\usepackage{manyfoot}%
\usepackage{booktabs}%
\usepackage{algorithm}%
\usepackage{algorithmicx}%
\usepackage{algpseudocode}%
\usepackage{listings}%
\usepackage{amsmath,amsfonts,amssymb}

\theoremstyle{thmstyleone}
\newtheorem{thm}{Theorem}[section]

\newtheorem{lemma}[thm]{Lemma}

\theoremstyle{remark}
\newtheorem{rem}{Remark}[section]
\numberwithin{equation}{section}
\begin{document}
\title[Certain functional identities on matrix rings]{Certain functional identities on matrix rings}

\author[1]{\fnm{Kaijia} \sur{Luo}}\email{kaijia\_luo@163.com}
\author[2]{\fnm{Jiankui} \sur{Li}}\email{jkli@ecust.edu.cn}

\affil[1]{\orgdiv{Institute of Mathematics}, \orgname{Hangzhou Dianzi University}, \orgaddress{\city{Hangzhou}, \postcode{310018}, \country{China}}}

\affil[2]{\orgdiv{School of Mathematics}, \orgname{East China University of Science and Technology}, \orgaddress{ \city{Shanghai}, \postcode{200237}, \country{China}}}

\abstract{
Let $D$ be a noncommutative division ring and let $R = M_{m}(D)$ with $m > 1$.
We characterize additive  mappings $f,$ $g:R\rightarrow R$ satisfying the identity $f(X) = X^{n}g(X^{-1})$ for every invertible element $X$ in $R$, where $n$ is a nonnegative integer.
We show that the solutions are precisely given by  $f = g$ and  $f(X) = Xf(I)$ for all $X$ in  $R$. 
Moreover, if $n \neq2$,  both $f$ and $g$ are identically zero.
}



\maketitle

\section{Introduction}

Let $R$ be an  associative ring with unity $I$. Denote by $R^{\times}$ the set of all invertible elements of $R$.
The problem concerning the continuity of an additive mapping $f$ on the field of real numbers satisfying $f(x)=x^2f(x^{-1})$ was posed by Halperin in 1963. 
It was later included by Aczél \cite{A} in his compilation of open questions in the theory of functional equations.
Two independent affirmative answers to Halperin's question were given by  Kurepa \cite{K} and  Jurkat \cite{J}.
Inspired by the work of Kurepa, Vukman~\cite{V1987} proved that every additive mapping  $f$ on a noncommutative division ring  $D$ of characteristic not 2 satisfying the identity $f(x)+x^2f(x^{-1})=0$ for all nonzero $x\in D$ must be identically zero.  
Dar and Jing~\cite{DJ} characterized the functional identity $f(x) = -x^{2}g(x^{-1})$ on a noncommutative division ring of characteristic not 2.
Subsequently,  Catalano and Merchán~\cite{CM} proved that on division rings of characteristic different from $2$ and $3$,  additive mappings $f$ and $g$ satisfying  $f(x) = -x^{n}g(x^{-1})$ for $n=3$ or $n=4$ are zero mappings.
Lee and  Lin \cite{LL} established the structure of additive mappings $f$ and $g$ satisfying $f(x) = x^{n}g(x^{-1})$ for all nonzero $x$ in a noncommutative division ring $D$, where $n\neq2$ is a positive integer and the characteristic of $D$ is not 2. 
Recently, Eroǧlu, Lee and Lin \cite{ELL} generalized the above results  to an arbitrary noncommutative division ring   as follows.

\begin{thm}\cite[Theorem 1.5]{ELL}
Let $D$ be a noncommutative division ring, and let $n$ be a nonnegative integer. Suppose that $f,$ $g : D \rightarrow D$ are additive mappings satisfying the identity $f(x) = x^{n}g(x^{-1})$ for all $x\in D^{\times}$. The following hold:\\
(i) If $n = 2$, then $f = g$ and $f(x) = xf(1)$ for all $x \in  D$.\\
(ii) If $n\neq 2$, then $f =g = 0$.
\end{thm}

Given a field $\mathbb{F}$ of characteristic not 2, the identity $f(x) = x^{n}g(x^{-1})$ for all $x\in \mathbb{F}^{\times}$ was completely determined by Ng \cite{N}.
We remark that Theorem 1.1 is in general not true if $D$ is a field of characteristic $p> 0$, see \cite[Example 4.3 (iii) and (iv)]{LL}.

For $n=2$, the characterization of the identity in Theorem 1.1 on $m\times m$ matrix rings over division rings was investigated by Dar and Jing \cite[Theorem 1.3]{DJ}. 
They proved that if additive mappings $f$ and $g$ on the matrix ring $R = M_m(D)$ (where $m>1$ and $D$ is a noncommutative division ring with $\operatorname{char}(D) \neq 2,3$) satisfy the identity $f(X) + X^{2}g(X^{-1}) = 0$ for all $X \in R^{\times}$, then there exists an element $P \in R$ such that $f(X) = XP$ and $g(X) = -XP$ for all $X \in R$.

Continuing this line of investigation, the goal of this paper is to extend Theorem 1.1   to the case of $m\times m$ matrix ring $R = M_{m}(D)$ over an arbitrary noncommutative division ring $D$, by considering the identity $f(X) = X^{n}g(X^{-1})$ for all invertible $X\in R^{\times}$. The main result  can be summarized as follows.

\begin{thm}
Let $D$ be a noncommutative division ring and let $R = M_{m}(D)$ with $m> 1$.
Suppose that additive mappings $f,$ $g :R \rightarrow R$ satisfy $f(X) = X^{n}g(X^{-1})$ for all  $X\in R^{\times}$, where $n$ is a nonnegative integer. \\
(i) If $n = 2$, then $f = g$ and $f(X) = Xf(I)$ for each $X \in  R$.\\
(ii) If $n \neq2$, then $f =g = 0$.
\end{thm}

\section{Theorem 1.2 with $\mathbf{char}\, D = 2$}

This section examines Theorem 1.2 in the setting where the characteristic of $D$ is 2.
Denote by $e_{ij}$  the $m\times m$ matrix whose $(i,j)$-entry is $1$ and all other entries are $0$.
Such matrices $e_{ij}$ (for $1\leq i,j \leq m$) are called the \emph{standard matrix units}.
Let $\pi_{ij} : M_{m}(D) \rightarrow D $ be the $(i,j)$-th coordinate projection.

We begin with the following result.
\begin{lemma}
Let $D$ be an arbitrary noncommutative division ring and  $R = M_{m}(D)$ with $m > 1$.
Suppose that  additive mappings $f$, $g :R \rightarrow R$ satisfy $f(X) = X^{n}g(X^{-1})$ for all  $X\in R^{\times}$, where $n$ is a nonnegative integer. \\
(i) If $n = 2$, then  $f(xI)=g(xI)= xf(I)$ for each $x\in  D$.\\
(ii) If $n \neq2$, then $f(xI)=g(xI) = 0$ for each $x\in  D$.
\end{lemma}
\begin{proof}
Let $a\in D^{\times}$.
By assumption, we have $f(aI) = (aI)^{n}g(a^{-1}I)$.
It follows that $e_{kk}f(aI)e_{ll} =a^{n}e_{kk} g(a^{-1}I)e_{ll}$ for all $1\leq k,l\leq m$.
Then the additive mappings $f_{kl}$, $g_{kl}:D\rightarrow D$ defined by $f_{kl}(x)=\pi_{kl}(f(xI))$ and $g_{kl}(x)=\pi_{kl}(g(xI))$ for each $x\in D$ satisfy  $f_{kl}(a)=a^{n}g_{kl}(a^{-1})$.
By Theorem 1.1, for each $x\in D$, we have  $f_{kl}(x) =g_{kl}(x)= xf_{kl}(1)$ for $n = 2$
and $f_{kl}=g_{kl} = 0$ for $n \neq2$.
It means that $e_{kk}f(xI)e_{ll} =e_{kk}g(xI)e_{ll}= xe_{kk}f(I)e_{ll}$ for $n = 2$ and $e_{kk}f(xI)e_{ll}=e_{kk}g(xI) e_{ll} = 0$ for $n \neq2$.
By the arbitrariness of $e_{kk}$ and $e_{ll}$, we have
\begin{equation}
f(xI) =g(xI)=
\begin{cases}
   xf(I),\quad  & if\quad n=2; \\
  0,  & if\quad n\neq 2. 
\end{cases}
\end{equation}
\end{proof}

In the following, we provide the  characterization of the identity $f(X) = X^{n}f(X^{-1})$.

\begin{lemma}
Let $D$ be a noncommutative division ring of characteristic 2  and let $R = M_{m}(D)$ with $m > 1$.
Suppose that an additive mapping $f :R \rightarrow R$ satisfies $f(X) = X^{n}f(X^{-1})$ for all  $X\in R^{\times}$, where $n$ is a nonnegative integer.\\
(i) If $n = 2$, then  $f(X) = Xf(I)$ for each $X \in  R$.\\
(ii) If $n \neq2$, then $f  = 0$.
\end{lemma}
\begin{proof}
Define the mapping $\tilde{f}:R\rightarrow R$ by $\tilde{f}(X)=f(X)-Xf(I)$ for all $X\in R$. 
 Then $\tilde{f}$ is an additive mapping and satisfies 
\begin{equation}
\tilde{f}(X) = X^{n}\tilde{f}(X^{-1})
\end{equation} for all  $X\in R^{\times}$. 
Indeed, if $n=2$, we have $ \tilde{f}(X) =f(X)-Xf(I)= X^{2}f(X^{-1})-Xf(I)=X^{2}(f(X^{-1})-X^{-1}f(I))=X^{2}\tilde{f}(X^{-1}).$
If $n \neq2$, it follows from Lemma 2.1 that $\tilde{f}=f$.
In particular, we have  
\begin{equation}
\tilde{f}(xI)=f(xI)-xf(I)=0
\end{equation}
  for each $x\in D$.

Let $a\in D^{\times}$ and $1\leq i\neq j\leq m$.

Suppose that $n$ is an odd number.  
Since $(I+ae_{ij})(I+ae_{ij})=I$ and $(I+ae_{ij})^{n}=I+ae_{ij}$, it follows from (2.2) that 
$\tilde{f}(I+ae_{ij})=(I+ae_{ij})^{n}\tilde{f}(I+ae_{ij})=(I+ae_{ij})\tilde{f}(I+ae_{ij}).$
Then we obtain 
\begin{equation}
ae_{ij}\tilde{f}(ae_{ij})=0.
 \end{equation}
 Substituting $a$ with $a^{-1}$ in   (2.4) yields
\begin{equation}
 a^{-1}e_{ij}\tilde{f}(a^{-1}e_{ij})=0.
 \end{equation}
Left-multiplying both sides of  (2.5) by $aI$, we have
\begin{equation}
e_{ij}\tilde{f}(a^{-1}e_{ij})=0.
\end{equation} 
Besides, using $(aI+ae_{ij})(a^{-1}I+a^{-1}e_{ij})=I$, we arrive at
\begin{equation}
\tilde{f}(aI+ae_{ij})=(aI+ae_{ij})^{n}\tilde{f}(a^{-1}I+a^{-1}e_{ij})
=a^{n}(I+e_{ij})\tilde{f}(a^{-1}e_{ij}).
\end{equation}
Combining   (2.3), (2.6) and  (2.7), we conclude that
\begin{equation}
\tilde{f}(ae_{ij})=a^{n}\tilde{f}(a^{-1}e_{ij}).
\end{equation}
  
Suppose that  $n$ is an  even number. 
Since  $(aI+ae_{ij})(a^{-1}I+a^{-1}e_{ij})=I$ and $(aI+ae_{ij})^{n}=a^{n}I$ , we arrive at
\begin{equation}
\tilde{f}(ae_{ij})=\tilde{f}(aI+ae_{ij})=(aI+ae_{ij})^{n}\tilde{f}(a^{-1}I+a^{-1}e_{ij})=a^{n}\tilde{f}(a^{-1}e_{ij}).
\end{equation} 
 
  Thus for each nonnegative integer $n$, we have 
   \begin{equation}
\tilde{f}(ae_{ij})=a^{n}\tilde{f}(a^{-1}e_{ij}).
\end{equation}
Hence  $\pi_{kl}(\tilde{f}(ae_{ij})) =a^{n}\pi_{kl}(\tilde{f}(a^{-1}e_{ij}))$ for all $1\leq k,l\leq m$. 
By Theorem 1.1,  we have $\pi_{kl}(\tilde{f}(xe_{ij}))=0$ for each $x\in D$,  which means that 
   \begin{equation}
\tilde{f}(xe_{ij})=0
\end{equation}
for each $x\in D$.

Next we claim $\tilde{f}(xe_{ii})=0$ for each $x\in D$.
Note that  $(I+(a-1)e_{ii})(I+(a^{-1}-1)e_{ii})=I$ and $(I+(a-1)e_{ii})^{n}=I+(a^{n}-1)e_{ii}$.
By (2.2) and (2.3), we have 
\begin{align}
\tilde{f}((a-1)e_{ii})&=\tilde{f}(I+(a-1)e_{ii})\nonumber\\
&=(I+(a-1)e_{ii})^{n}\tilde{f}(I+(a^{-1}-1)e_{ii})\nonumber\\
&=(I+(a^{n}-1)e_{ii})\tilde{f}((a^{-1}-1)e_{ii}).
\end{align}
Besides, using $(aI+(1-a)e_{ii})(a^{-1}I+(1-a^{-1})e_{ii})=I$ and $(aI+(1-a)e_{ii})^{n}=a^{n}I+(1- a^{n})e_{ii}$,
we have
\begin{align}
\tilde{f}((1-a)e_{ii})&=\tilde{f}(aI+(1-a)e_{ii})\nonumber\\
&=(aI+(1-a)e_{ii})^{n}\tilde{f}(a^{-1}I+(1-a^{-1})e_{ii})\nonumber\\
&=(a^{n}I+(1- a^{n})e_{ii})\tilde{f}((1-a^{-1})e_{ii}).
\end{align}
Adding  (2.12)  and (2.13),  we have 
\begin{equation}
(a^{n}-1)\tilde{f}((a^{-1}-1)e_{ii})=0.
\end{equation}
Thus, left-multiplying both sides of  (2.14) by $e_{ii}$, we get 
\begin{equation}
(a^{n}-1)e_{ii}\tilde{f}((a^{-1}-1)e_{ii})=0.
\end{equation}
Combining  (2.12)  and (2.15), we have  $\tilde{f}((a-1)e_{ii})=\tilde{f}((a^{-1}-1)e_{ii})$, which is equivalent to $\pi_{kl}(\tilde{f}(ae_{ii}))=\pi_{kl}(\tilde{f}(a^{-1}e_{ii}))$ for all $1\leq k,l\leq m$. 
By Theorem 1.1, we have  $\pi_{kl}(\tilde{f}(xe_{ii}))=0$ for all $x\in D$ and  $1\leq k,l\leq m$.
Hence 
\begin{equation}
\tilde{f}(xe_{ii})=0.
\end{equation}

Therefore, in view of  (2.11)  and (2.16), we have $\tilde{f}=0$.
By the definition of $\tilde{f}$,
\begin{equation*}
f(X) =
\begin{cases}
   Xf(I),\quad  & if\quad n=2, \\
  0,  & if\quad n\neq 2. 
\end{cases}
\end{equation*} for each $X\in R$.
\end{proof}

We now turn to the proof of Theorem 1.2 for the case of characteristic~2, employing the same method as in \cite[Lemma 2.1]{ELL}.

\begin{thm}
Let $D$ be a noncommutative division ring of characteristic 2 and let $R = M_{m}(D)$ with $m > 1$.
Suppose that additive mappings $f,$ $g :R \rightarrow R$ satisfy $f(X) = X^{n}g(X^{-1})$ for all  $X\in R^{\times}$, where $n$ is a nonnegative integer.\\
(i) If $n = 2$, then $f = g$ and $f(X) = Xf(I)$ for each $X \in  R$.\\
(ii) If $n \neq2$, then $f =g = 0$.
\end{thm}
\begin{proof}
It is obvious that $g(X) = X^{n}f(X^{-1})$ for all $X\in R^{\times}$. Thus
$$(f + g)(X) = X^{n}(f + g)(X^{-1})$$
for all $X\in R^{\times}$. 

If $n = 2$, it follows from Lemma 2.2 that $(f + g)(X) = X(f + g)(I)$ for all $X\in R$. 
Note that $f(I)=g(I)$. Then  $(f + g)=2 Xf (I)=0$. Hence $f=g$. Using  Lemma 2.2 again, we obtain  $f(X) = Xf(I)$ for all $X \in  R$.

If $n \neq 2$,  it follows from Lemma 2.2 that  $f+g=0$. Then $f=g$. Hence $f=0$.
\end{proof}

\section{Theorem 1.2 with $\mathbf{char}\, D \neq 2$}

This section is dedicated to proving Theorem 1.2 for a noncommutative division ring of characteristic not 2.
We begin with the special case where $f = g$ in the identity $f(X) = X^{n}g(X^{-1})$.

\begin{lemma}
Let $D$ be a noncommutative division ring of characteristic not 2 and let $R = M_{m}(D)$ with $m > 1$.
Suppose that an additive mapping $f :R \rightarrow R$ satisfies $f(X) = X^{n}f(X^{-1})$ for all  $X\in R^{\times}$, where $n$ is a nonnegative integer. Then $f(X)=Xf(I)$ for all $X\in R$. In particular, $f=0$  for $n\neq2$.
\end{lemma}
\begin{proof}
Let $x\in D$  and $1\leq i\neq j\leq m$.

In view of Lemma 2.1, $f(xI)=xf(I)$ for $n=2$ and   $f(xI)=0$ for $n\neq2$.
Define the mapping $\tilde{f}:R\rightarrow R$ by $\tilde{f}(X)=f(X)-Xf(I)$ for all $X\in R$. Then $\tilde{f}$ is an additive mapping and satisfies 
\begin{equation}
\tilde{f}(X) = X^{n}\tilde{f}(X^{-1})
\end{equation} for all  $X\in R^{\times}$. 
We obtain
\begin{equation}
\tilde{f}(xI)=f(xI)-xf(I)=0.
\end{equation}

Since $(I+xe_{ij})(I-xe_{ij})=I$ and $(I+xe_{ij})^{n}=(I+nxe_{ij})$, it follows from (3.1) that 
\begin{eqnarray*}
\tilde{f}(xe_{ij})= \tilde{f}(I+xe_{ij})= (I+xe_{ij})^{n}\tilde{f}(I-xe_{ij})=-(I+nxe_{ij})\tilde{f}(xe_{ij}),
\end{eqnarray*}
which means that
 \begin{equation}
2\tilde{f}(xe_{ij})+nxe_{ij}\tilde{f}(xe_{ij})=0.
\end{equation}
Replacing $x$ with $-x$ in  (3.3) and comparing with  (3.3), we have 
 \begin{equation}
\tilde{f}(xe_{ij})=0
\end{equation}
for all  $x\in D$.

Let $a\in D^{\times}$.
Replacing the invertible element $X$ in (3.1) with $I+(a-1)e_{ii}$ and $aI+(1-a)e_{ii}$ in turn, we obtain
\begin{align*}
&\tilde{f}(I+(a-1)e_{ii})=(I+(a-1)e_{ii})^{n}\tilde{f}(I+(a^{-1}-1)e_{ii}),\\
&\tilde{f}(aI+(1-a)e_{ii})=(aI+(1-a)e_{ii})^{n}\tilde{f}(a^{-1}I+(1-a^{-1})e_{ii}).
\end{align*}
Applying  (3.2), these expressions simplify to
\begin{align}
&\tilde{f}((a-1)e_{ii})=(I+(a^{n}-1)e_{ii})\tilde{f}((a^{-1}-1)e_{ii}),\\
&\tilde{f}((a-1)e_{ii})=(a^{n}I+(1-a^{n})e_{ii})\tilde{f}((a^{-1}-1)e_{ii}).
\end{align}
Comparing  (3.5) and  (3.6), we have 
\begin{equation}
0=[(1-a^{n})I+2(a^{n}-1)e_{ii}]\tilde{f}((a^{-1}-1)e_{ii})
\end{equation}
Left-multiplying both sides of  (3.7) by $e_{ii}$, we have 
\begin{equation}
0=(a^{n}-1)e_{ii}\tilde{f}((a^{-1}-1)e_{ii})
\end{equation}
By (3.5) and (3.8), we obtain
\begin{align}
\tilde{f}((a-1)e_{ii})=\tilde{f}((a^{-1}-1)e_{ii}),
\end{align}
which means that 
$$\tilde{f}(ae_{ii})=\tilde{f}(a^{-1}e_{ii}),$$
which is equivalent to $\pi_{kl}(\tilde{f}(ae_{ii}))=\pi_{kl}(\tilde{f}(a^{-1}e_{ii}))$ for all $1\leq k,l\leq m$. 
It follows from  Theorem 1.1 that $\pi_{kl}(\tilde{f}(xe_{ii}))=0$. Hence, we have
\begin{align}
\tilde{f}(xe_{ii})=0
\end{align} 
for each  $x\in D$.

It follows from   (3.4) and  (3.10) that $\tilde{f}(X)=0$ for each $X\in R$. 
Hence $f(X)=Xf(I)$ for all $X\in R$. In particular, $f=0$  for $n\neq2$.
\end{proof}

The following conclusion plays a key role in the proof of the main theorem.

\begin{lemma}
Let $D$ be a noncommutative division ring of characteristic not 2 and let  $R = M_{m}(D)$ with $m > 1$.
Suppose that an additive mapping $f :R \rightarrow R$ satisfies $f(X) + X^{n}f(X^{-1})=0$ for all  $X\in R^{\times}$, where $n$ is a nonnegative integer. Then $f=0$.
\end{lemma}
\begin{proof}
Note that $f(I)=0$. By Lemma 2.1, we have $f(xI)=0$ for all  $x\in D$. 

Let $a\in D^{\times}$ and $1\leq i\neq j\leq m$.

By assumption, using the invertible element $I+ae_{ij}$, we have 
\begin{equation}
 f(I+ae_{ij})+(I+ae_{ij})^{n}f(I-ae_{ij})=0,
\end{equation}
which implies $nae_{ij}f(ae_{ij})=0.$ It follows that
\begin{equation}
ne_{ij}f(a^{-1}e_{ij})=0.
\end{equation}
Since $(aI+ae_{ij})(a^{-1}I-a^{-1}e_{ij})=I$, we have 
\begin{align}
0&=f(aI+ae_{ij})+(aI+ae_{ij})^{n}f(a^{-1}I-a^{-1}e_{ij}) \nonumber  \\
  &=f(ae_{ij})-a^{n}(I+ne_{ij})f(a^{-1}e_{ij}) \nonumber \\
  &=f(ae_{ij})-a^{n}f(a^{-1}e_{ij}),
\end{align}
which  means that $$\pi_{kl}(f(ae_{ij}))=a^{n}\pi_{kl}(f(a^{-1}e_{ij}))$$ for all $a\in D^{\times}$ and $1\leq k,l\leq m$.
By Theorem 1.1, we have 
\begin{equation}
f(xe_{ij})=0
\end{equation}
for each $x\in D$.

Using invertible elements $I+(a-1)e_{ii}$, $aI+(1-a)e_{ii}$, we have
\begin{align*}
&{f}(I+(a-1)e_{ii})+(I+(a-1)e_{ii})^{n}{f}(I+(a^{-1}-1)e_{ii})=0,\\
&{f}(aI+(1-a)e_{ii})+(aI+(1-a)e_{ii})^{n}{f}(a^{-1}I+(1-a^{-1})e_{ii})=0.
\end{align*}
Since $f(xI)=0$ for each $x\in D$, we have 
\begin{align}
&{f}((a-1)e_{ii})+(I+(a^{n}-1)e_{ii}){f}((a^{-1}-1)e_{ii})=0,\\
&{f}((a-1)e_{ii})+(a^{n}I+(1-a^{n})e_{ii}){f}((a^{-1}-1)e_{ii})=0.
\end{align}
Comparing   (3.15) and  (3.16), we have 
\begin{equation}
[(1-a^{n})I+2(a^{n}-1)e_{ii}]{f}((a^{-1}-1)e_{ii})=0.
\end{equation}
Left-multiplying both sides of  (3.17) by $e_{ii}$, we have 
\begin{equation}
(a^{n}-1)e_{ii}{f}((a^{-1}-1)e_{ii})=0
\end{equation}
Applying  (3.18) to  (3.15), we obtain
\begin{align}
{f}((a-1)e_{ii})+{f}((a^{-1}-1)e_{ii})=0,
\end{align}
which means that 
$${f}(ae_{ii})+{f}(a^{-1}e_{ii})=2{f}(e_{ii}).$$
Taking $ae_{ii}=2e_{ii}$, we have ${f}(2e_{ii})+\frac{1}{2}{f}(e_{ii})=2{f}(e_{ii})$.
So ${f}(e_{ii})=0$. It follows that ${f}(ae_{ii})+{f}(a^{-1}e_{ii})=0$
for all $a\in D^{\times}$,
which  means that $$\pi_{kl}({f}(ae_{ii}))=-\pi_{kl}({f}(a^{-1}e_{ii}))$$ for all $1\leq k,l\leq m$.
In view of  Theorem 1.1, we have
\begin{align}
f(xe_{ii})=0
\end{align} 
for each  $x\in D$.

It follows from  (3.14) and  (3.20) that  $f=0$.
\end{proof}

Combining  Lemmas 3.1 and 3.2, we have the following result.
\begin{thm}
Let $D$ be a noncommutative division ring of characteristic not 2 and let $R = M_{m}(D)$ with $m > 1$.
Suppose that additive mappings $f,$ $g :R \rightarrow R$ satisfy $f(X) = X^{n}g(X^{-1})$ for all  $X\in R^{\times}$, where $n$ is a nonnegative integer.\\
(i) If $n = 2$, then $f = g$ and $f(X) = Xf(I)$ for each $X \in  R$.\\
(ii) If $n \neq2$, then $f =g = 0$.
\end{thm}
\begin{proof}
  By assumption, we have $g(X) = X^{n}f(X^{-1})$ for all $X\in R^{\times}$.
It follows that 
$$(f + g)(X) = X^{n}(f + g)(X^{-1}),$$
$$(f -g)(X) = -X^{n}(f - g)(X^{-1})$$
for all $X\in R^{\times}$. 
It follows from Lemma  3.1 that $$(f+g)(X) = X(f+g)(I)$$ for all $X \in  R$ and $f+g=0$ for $n\neq2$.
In view of Lemma  3.2, we have $$f-g=0.$$
Therefore, we conclude that  $f(X) = g(X)=Xf(I)$ for all $X \in  R$. 
In particular, when  $n \neq2$, we have  $f =g = 0$.
\end{proof}

\begin{rem} Theorem 1.2 generalizes Theorem 1.1, with the latter corresponding to the special case 
$m=1$.
\end{rem}

We complete the  whole proof of Theorem 1.2.

\end{document}